\documentclass[12pt,reqno]{article} 

\usepackage{amsmath,amssymb,amsfonts,amsthm}

\newif\ifarrows 
\arrowsfalse

\usepackage[left=3cm,top=3cm,right=3cm,bottom=3cm]{geometry}
\usepackage[usenames,dvipsnames]{xcolor}
\usepackage{tikz}
\ifarrows
\fi

\usepackage{subfigure}
\usepackage[capitalize]{cleveref}

\newcommand{\remove}[1]{}

 \usepackage[normalem]{ulem} 

\begin{document}
\newtheorem{theorem}{Theorem}[section]
\newtheorem{lemma}[theorem]{Lemma}
\newtheorem{proposition}[theorem]{Proposition}
\newtheorem{corollary}[theorem]{Corollary}
\newtheorem{remark}[theorem]{Remark}

\theoremstyle{definition}
\newtheorem{defn}[theorem]{Definition}

\newcommand{\N}{{\mathbb N}}
\newcommand{\R}{{\mathbb R}}

\newcommand{\E}{\mathbb E}
\newcommand{\V}{\mathrm{Var}}
\newcommand{\Prob}{\mathbb{P}}

\newcommand{\Ind}{\ensuremath{\textbf{1}}}
\newcommand{\bo}{\ensuremath{\mathrm{O}}}

\newcommand{\GM}{\mathfrak{D}}	
\newcommand{\RG}{\mathcal{I}}	

\newcommand{\ds}[1]{\ensuremath{\textbf{d}_{#1}}}
\newcommand{\md}[2]{\ensuremath{m_{#1}({#2})}} 
\newcommand{\dsec}{\ensuremath{d^{(2)}}}	

\newcommand{\SL}{\mathfrak{s}}	
\newcommand{\TL}{\mathfrak{t}}	
\newcommand{\CL}{\mathfrak{c}}	

\newcommand{\ff}[2]{\langle #1 \rangle_{#2}}	
\newcommand{\comp}[2]{{#1}^{#2}}	

\title{A limit theorem for the six-length of random functional graphs with a fixed degree sequence } 
\author{ Kevin Leckey\thanks{Current address: Fakult\"at Statistik,  Technische Universit\"at Dortmund, Germany}\phantom{z}\ and  Nicholas Wormald\thanks{Supported by an ARC Australian Laureate Fellowship.}  \\ School of Mathematical Sciences \\ Monash University
} 

\maketitle

 \begin{abstract}
We obtain results on the limiting distribution of the six-length of a random functional graph, also called a   functional digraph or random mapping, with given in-degree sequence. The six-length  of a vertex $v\in V$  is defined from the associated mapping, $f:V\to V$, to be the maximum $i\in V$ such that the elements $v, f(v), \ldots, f^{i-1}(v)$ are all distinct.  This has relevance to the study of algorithms for integer factorisation. 
\end{abstract}

\section{Introduction}\label{sec:intro}

We consider random directed graphs with all out-degrees equal to 1, which we call {\em functional graphs} (see Section~\ref{s:defs} for further notation) or random mappings. 
The motivation in most of the related literature is a better understanding of Pollard's $\rho$-algorithm \cite{Po75} for integer factorisation, or the improved version by Brent and Pollard \cite{BrPo81}. The runtime depends on the six-length (also called $\rho$-length) of a polynomial in $\mathbb{F}_p [x]$. (Pollard's first version used $x^2-1$.) Under the assumption that a polynomial mod $p$ `behaves like' a random mapping (supported by some research listed below), we are interested in the six-length of random mappings in particular.  Martins and Panario~\cite{MP16} studied polynomials in $\mathbb{F}_p [x]$, in particular the six-length  in several random models. They found significance in the six-length of random polynomials with given in-degree sequence, and gave numerical results for several random models. Our main aim is to derive results on the six-length of random functional graphs with given in-degree sequence, to give a baseline for comparison with random polynomial models.

 Results pertinent to our study were obtained by Arney and Bender~\cite{AB82}, who were motivated by the study of random shift registers.  For a fixed  set $\mathcal{D}$, they considered a functional graph chosen uniformly at random among those with in-degrees in $\mathcal{D}$. They studied various properties such as the in-degrees of vertices, tree size, tail length and  six-length. They also obtained some information on the number of origins (vertices of in-degree 0), stopping short of being able to specify the number of origins. 
 Hansen and Jaworski~\cite{HJ08} considered a two-stage experiment: (1) Choose random indegrees $D_1,\ldots,D_n$ from an exchangeable probability distribution, (2) Choose a functional graph at random among graphs with indegrees $D_1,\ldots,D_n$.  They studied the number of cyclic vertices (vertices lying on a cycle) and of components, and component sizes.\remove{This was extended   in~\cite{HJ14} with regard to a (non-uniform) model that restricts indegrees to the values 1 and 2, but again falls well short of being able to specify the in-degree sequence.}
    
Our main results are stated in Section~\ref{s:defs} after some basic definitions. In particular we give the limiting distribution of the six-length for functional graphs with given indegree sequence, and also asymptotics for the moments of the distribution, as well as the joint distribution of the tail- and six-lengths. Proofs for the case that the second moment of the indegree sequence is ``large" are given in Section~\ref{sec_proofs}, and for the remaining case (except for some almost trivial cases) in Section~\ref{sec_small_sig}. See also Konyagin, Luca,  Mans, Mathieson, Sha~and Shparlinski~\cite{KLMSS16} for a study of polynomials over finite fields considering similar aspects, such as largest component and tree size of the associated functional digraphs. Similar to~\cite{MP16}, they  observe, in~\cite[Section 4]{KLMSS16}, that the in-degree sequence of these random digraphs is  distributed rather differently from that of uniformly random functional digraphs.

\section{Definitions, model and results} \label{s:defs}

\noindent
{\bf Functional Graphs.}
The functional graph of a function $f:V\rightarrow V$ is a directed graph $\mathcal{G}_f$ with vertex set $V$ and edge set $\{ (v,f(v)): v\in V\}$.  Consider, for example, the vertex set $V=\{0,\ldots, 4\}$ and the function $f(x)=x^2$ (mod $5$). Then $\mathcal{G}_f$ is given by   

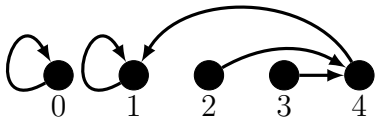
\begin{figure}[h!]\centering

\begin{tikzpicture}[every edge/.append style={-{latex},very thick}] 

\node[fill=black,circle,draw] (1) at (0,0){};
\node[fill=black,circle,draw] (2) at (1,0){};
\node[fill=black,circle,draw] (3) at (2,0){};
\node[fill=black,circle,draw] (4) at (3,0){};
\node[fill=black,circle,draw] (5) at (-1,0){};

\path (1) edge [ out=210,in=120,looseness=9] (1);
\path (2) edge [out=30,in=150] (4);
\path (3) edge (4);
\path (4) edge [out=120,in=60] (1);
\path (5) edge [out=210,in=120,looseness=9] (5);

\node[draw=none,fill=none] (lab1) at (0,-0.4){$1$};
\node[draw=none,fill=none] (lab2) at (1,-0.4){$2$};
\node[draw=none,fill=none] (lab3) at (2,-0.4){$3$};
\node[draw=none,fill=none] (lab4) at (3,-0.4){$4$};
\node[draw=none,fill=none] (lab5) at (-1,-0.4){$0$};
\end{tikzpicture}
\end{figure}

\noindent
The six-length of a vertex in a functional graph is defined as follows:
Let $f:V\rightarrow V$ be a function and let $id$ denote the identity function on $V$. Let $\comp f k$ denote the $k$-times composition of $f$, that is $\comp f 0 =id$ and $\comp f k= \comp f {k-1} \circ f$ for $k\geq 1$.
The {\em six-length} of $v\in V$ is defined as
$$\SL_f(v)=\min\left\{ k\in\N : \comp f k (v) \in \{\comp f j (v) : 0\leq j\leq k-1\}\right\}.$$
An example for the six-length in a functional graph is given in \cref{fig:sixlength}. Note that $\SL_f(v)$ can be decomposed into the tail-length $\TL_f(v)$ and the cycle-length $\CL_f(v)$ as indicated in \cref{fig:TL_CL}. More formally, the tail-length is the unique integer that satisfies
\begin{align*}
\TL_f(v)<\SL_f(v)\quad \text{ and }\quad  \comp f {\SL_f(v)} (v)=\comp f {\TL_f(v)} (v),
\end{align*}
and the cycle-length is given by $\CL_f(v):=\SL_f(v)-\TL_f(v)$.\\

\begin{figure}\centering
\subfigure[{\color{red} $\SL(v)=11$}\label{fig:sixlength}]{
\begin{tikzpicture}[every edge/.append style={-{latex},very thick}] 

\node[fill=red,circle,draw] (a) at (1.4,3.4){};
\node[fill=red,circle,draw] (b) at (0.58,2.82){};
\node[fill=red,circle,draw] (c) at (0,2){};
\node[fill=red,circle,draw] (1) at (0,1){};
\node[fill=red,circle,draw] (2) at (0,0){};
\node[fill=red,circle,draw] (3) at (0.7,-0.7){};
\node[fill=red,circle,draw] (4) at (1.7,-0.7){};
\node[fill=red,circle,draw] (5) at (2.4,0){};
\node[fill=red,circle,draw] (6) at (2.4,1){};
\node[fill=red,circle,draw] (7) at (1.7,1.7){};
\node[fill=red,circle,draw] (8) at (0.7,1.7){};

\node[draw=none,fill=none] (lab) at (1.4,3){${\color{red} v}$};

\path (a) edge [color=red] (b);
\path (b) edge [color=red] (c);
\path (c) edge [color=red] (1);
\path (1) edge [color=red] (2);
\path (2) edge [color=red] (3);
\path (3) edge [color=red] (4);
\path (4) edge [color=red] (5);
\path (5) edge [color=red] (6);
\path (6) edge [color=red] (7);
\path (7) edge [color=red] (8);
\path (8) edge [color=red] (1);

\node[fill=black,circle,draw] (v1) at (2.4,3.4){};
\path (v1) edge (a);
\node[fill=black,circle,draw] (v2) at (-0.12,3.52){};
\path (v2) edge (b);
\node[fill=black,circle,draw] (v3) at (-0.42,2.82){};
\path (v3) edge (b);

\node[fill=black,circle,draw] (v4) at (2.6,1.8){};
\path (v4) edge (7);
\node[fill=black,circle,draw] (v5) at (3.3,1.1){};
\path (v5) edge (v4);
\node[fill=black,circle,draw] (v6) at (3.3,2.5){};
\path (v6) edge (v4);

\node[fill=black,circle,draw] (v7) at (2.92,-0.82){};
\path (v7) edge (5);
\node[fill=black,circle,draw] (v8) at (3.5,0){};
\path (v8) edge (v7);
\end{tikzpicture}
}
\hspace{1cm}
\subfigure[{\color{blue}$\TL(v)=3$}, {\color{ForestGreen}$\CL(v)=8$}\label{fig:TL_CL}]{
\begin{tikzpicture}[every edge/.append style={-{latex},very thick}] 

\node[fill=blue,circle,draw] (a) at (1.4,3.4){};
\node[fill=blue,circle,draw] (b) at (0.58,2.82){};
\node[fill=blue,circle,draw] (c) at (0,2){};
\node[fill=ForestGreen,circle,draw] (1) at (0,1){};
\node[fill=ForestGreen,circle,draw] (2) at (0,0){};
\node[fill=ForestGreen,circle,draw] (3) at (0.7,-0.7){};
\node[fill=ForestGreen,circle,draw] (4) at (1.7,-0.7){};
\node[fill=ForestGreen,circle,draw] (5) at (2.4,0){};
\node[fill=ForestGreen,circle,draw] (6) at (2.4,1){};
\node[fill=ForestGreen,circle,draw] (7) at (1.7,1.7){};
\node[fill=ForestGreen,circle,draw] (8) at (0.7,1.7){};

\node[draw=none,fill=none] (lab) at (1.4,3){${v}$};

\path (a) edge [color=blue] (b);
\path (b) edge [color=blue] (c);
\path (c) edge [color=blue] (1);

\path (1) edge [color=ForestGreen] (2);
\path (2) edge [color=ForestGreen] (3);
\path (3) edge [color=ForestGreen] (4);
\path (4) edge [color=ForestGreen] (5);
\path (5) edge [color=ForestGreen] (6);
\path (6) edge [color=ForestGreen] (7);
\path (7) edge [color=ForestGreen] (8);
\path (8) edge [color=ForestGreen] (1);

\node[fill=black,circle,draw] (v1) at (2.4,3.4){};
\path (v1) edge (a);
\node[fill=black,circle,draw] (v2) at (-0.12,3.52){};
\path (v2) edge (b);
\node[fill=black,circle,draw] (v3) at (-0.42,2.82){};
\path (v3) edge (b);

\node[fill=black,circle,draw] (v4) at (2.6,1.8){};
\path (v4) edge (7);
\node[fill=black,circle,draw] (v5) at (3.3,1.1){};
\path (v5) edge (v4);
\node[fill=black,circle,draw] (v6) at (3.3,2.5){};
\path (v6) edge (v4);

\node[fill=black,circle,draw] (v7) at (2.92,-0.82){};
\path (v7) edge (5);
\node[fill=black,circle,draw] (v8) at (3.5,0){};
\path (v8) edge (v7);
\end{tikzpicture}
}
\caption{An illustration of the six-, tail-, and cycle-length.}
\end{figure}
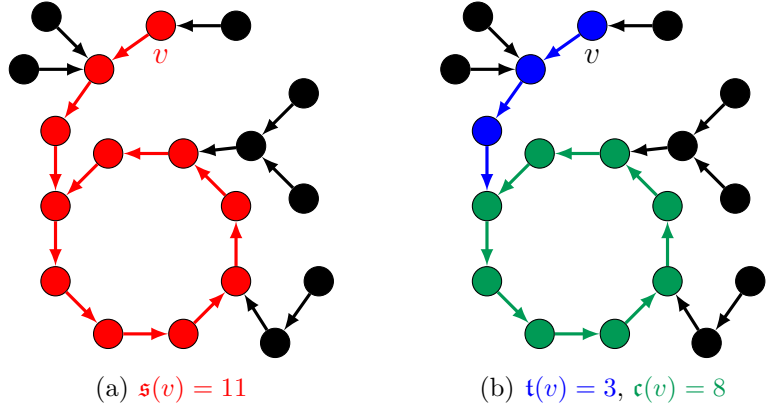

\noindent
{\bf Random Model.} Throughout the paper, a (finite) sequence $\ds n =(d_{n,1},\ldots,d_{n,n})$ is called degree sequence if
\begin{align}
\sum_{j=1}^n d_{n,j}&=n\quad \text{and}\quad \ds n \in\N_0^n. \label{a0}\tag{A0}
\end{align}
A random functional graph with degree sequence $\ds n$ is a graph $\mathcal{G}_F$ where $F$ is drawn uniformly at random from the set
\begin{align}\label{func_set}
\mathfrak{F}(\ds n):=\left\{f:[n]\rightarrow [n] : |f^{-1}(\{i\})|=d_{n,i}\text{ for all }i\in[n]\right\}.
\end{align}
 Here and elsewhere, we use  $[n]:=\{1,\ldots,n\}$. 
Note that technically what we call the degree sequence is the indegree sequence of the directed graph. This simplification is sensible because all outdegrees are 1.

Now let $\{\ds n : n\in\N\}$ be a family of degree sequences. Let $\SL_n(v)$ and $\TL_n(v)$ be six- and tail-length of a vertex $v\in[n]$ in a random functional graph with degree sequence $\ds n$. 
The aim of this paper is to investigate the asymptotic behaviour of $(\SL_n(v),\TL_n(v))$. 

We use the usual asymptotic  notation such as $\bo, \Omega, \Theta, o, \omega, \sim$;  in particular 
$a_n=\omega(b_n)$ if $b_n=o(a_n)$. Also, for any positive integers $n,k\in \N$ with $k\leq n$ let 
\begin{align*}
\ff n k:=n! /(n-k)!\,.
\end{align*}

\noindent
{\bf Degree sequences.}
For a degree sequence $\ds n =(d_{n,1},\ldots,d_{n,n})$ let
\begin{align}\label{def_Del_m_sigma}
\Delta(\ds n):=\max_j d_{n,j},\quad
\md k {\ds n}:=\sum_{j=1}^n d_{n,j}^k,\quad
\sigma^2(\ds n):=\frac {\md 2 {\ds n}} n  -1.
\end{align}
The parameter $\sigma^2(\ds n)$ is sometimes called the {\it coalescence}.

Throughout this section, let  $\{\ds n : n\in\N\}$ be a family of degree sequences and let $\{v_n: n\in\N\}$ be a family of vertices with $v_n\in [n]$. For the upcoming limit theorem for $\SL_n(v_n)$ we assume the following: 
\begin{align}
\sigma^2(\ds n)&=o(n)\quad\text{and}\quad \sigma^2(\ds n)=\omega\left(n^{-1}\right),\label{a1}\tag{A1}\\
\Delta( {\ds n})&=o\left(\sqrt{n \sigma^2(\ds n)}\right),\label{a2}\tag{A2}
\end{align}
\begin{theorem}\label{thm_sixlength}
Assume \eqref{a0}, \eqref{a1} and \eqref{a2}.
 Then $\left(\SL_n(v_n)/ \sqrt{n /\sigma^2(\ds n)}\right)_{n\geq 1}$ converges weakly to the   standard Rayleigh distribution, that is
\begin{align*}
\lim_{n\rightarrow\infty} \Prob\left(\SL_n(v_n)> x\sqrt{n /\sigma^2(\ds n)}\right)=\mathrm{e}^{-x^2/ 2},\quad x>0.
\end{align*}
\end{theorem}
In fact the methods used to prove Theorem \ref{thm_sixlength} 
also yield the convergence of all moments for a wide range of degree sequences. More precisely, let 
\begin{align}
\sigma^2(\ds n)&=o\left(\frac n {(\log n)^3} \right)\quad\text{and}\quad \sigma^2(\ds n)=\omega\left(n^{-1}\right),\label{b1}\tag{B1}\\
\Delta (\ds n)&=o\left(\sqrt{\frac{n\sigma^2(\ds n)}{(\log n)^{3}}}\right).\label{b2}\tag{B2}
\end{align}
Then the convergence in Theorem \eqref{thm_sixlength} also holds with respect to all moments, that is:
\begin{theorem}\label{thm_moments}
Assume \eqref{a0}, \eqref{b1} and \eqref{b2}. Let $X$ be  standard Rayleigh distributed. Then
\begin{align*}
\lim_{n\rightarrow\infty} \E\left[\left(\frac{\SL_n(v_n)}{\sqrt{n /\sigma^2(\ds n)}}\right)^p\right]=\E[X^p],\quad p\geq 1.
\end{align*}
In particular, $\E[\SL_n(v_n)]\sim \sqrt{\frac {\pi n} {2\sigma^2(\ds n)}}$ and $\V(\SL_n(v_n))\sim \frac{4-\pi} {2\sigma^2(\ds n)} n $.
\end{theorem}
Moreover, these assumptions also imply that the ratio between tail-length and six-length is asymptotically uniformly distributed. More precisely:
\begin{theorem}\label{thm_taillength} Let $X$ and $U$ be independent, $U$ be uniformly distributed on $[0,1]$, and $X$ be Rayleigh distributed. 
Assume \eqref{a0}, \eqref{b1} and \eqref{b2}. Then
\begin{align*}
\left(\frac{\SL_n(v_n)}{\sqrt{n /\sigma^2(\ds n)}}\, ,\,\frac{\TL_n(v_n)}{\sqrt{n /\sigma^2(\ds n)}}\right)\stackrel{d}{\longrightarrow} (X, UX).
\end{align*}
\end{theorem}
\begin{remark} A combination of \cref{thm_moments} and \cref{thm_taillength} yields
\begin{align*}
\E[\TL_n(v_n)]\sim \sqrt{ \frac{\pi n } {8\sigma^2(\ds n)}}\quad\text{and}\quad \E[\CL_n(v_n)]\sim\sqrt{ \frac{\pi n } {8\sigma^2(\ds n)}}.
\end{align*}
These results support 
 a conjecture by Brent and Pollard \cite[Section 3]{BrPo81} on the typical tail- and cycle-length of polynomials mod $p$.
\end{remark}

\section{Proofs for sequences with large coalescence}\label{sec_proofs}
 
We first prove all Theorems under the additional assumption
\begin{align}
\sigma^2(\ds n)=\omega\left(\frac {\log n} {n^{1/3}}\right).\label{a+}\tag{A+}
\end{align}
Cases with $\sigma^2(\ds n)=\bo\left(\log n / n^{1/3}\right)$ will be discussed in \cref{sec_small_sig}.
 
Throughout this section we omit the dependence on $\ds n$ in the notation. In particular
\begin{align*}
\Delta:=\Delta(\ds n),\quad m_j:=\md j {\ds n},\quad \sigma^2:=\sigma^2(\ds n).
\end{align*}
Moreover, we also omit the dependence on $n$ in the notation of the degrees, that is 
\begin{align*}
(d_1,\ldots,d_n):=(d_{n,1},\ldots,d_{n,n}).
\end{align*}
Unless stated otherwise, $n$ is a positive integer and asymptotic results are as \mbox{$n\rightarrow\infty$}. Condition \eqref{a0} is the only condition assumed throughout the section. All other assumptions are stated in the lemmas separately.
\subsection{Limit theorem for the six-length}  This section contains the proof of Theorem~\ref{thm_sixlength} for degree sequences that additionally satisfy \eqref{a+}, that is we prove the following statement:
\begin{proposition}\label{thm_sl_a+}
Assume \eqref{a0}, \eqref{a1}, \eqref{a2} and \eqref{a+}.
 Then 
\begin{align*}
\lim_{n\rightarrow\infty} \Prob\left(\SL_n(v_n)> x\sqrt{n /\sigma^2(\ds n)}\right)=\mathrm{e}^{-x^2/ 2},\quad x>0.
\end{align*}
\end{proposition}
The proof of  is based on the following explicit formula for the probabilities. In fact, the formula below remains valid even without making any assumptions on the degree sequence other than \eqref{a0}.
\begin{lemma}\label{lem_dist_SL}
For every $n\geq 2$ and  $v\in [n]$
\begin{align*}
\Prob(\SL_n(v) > k) = \frac {1} {\ff n k} \sum_{(i_1,\ldots, i_k)\in J_{n,k}(v)} \prod_{j=1}^k d_{i_j},\quad 1\leq k\leq n-1,
\end{align*}
with $\ff n k=\prod_{j=0}^{k-1} (n-j)$ and $J_{n,k}(v)=\{ (j_1,\ldots,j_k)\in ([n] \setminus \{v\})^k : j_\ell \neq j_m \text{ for } \ell \neq m\}$.
\end{lemma}
\begin{proof}
Recall that $F$ denotes a function drawn uniformly at random from the set $\mathfrak{F}(\ds n)$ defined in \eqref{func_set}.
Note that $J_{n,k}(v)$ corresponds to the set of all possible non-self-intersecting $k$-paths starting at $v$. Thus, we have
\begin{align*}
\Prob(\SL_n(v) > k)=\sum_{J\in J_{n,k}(v)}
\Prob\left(\left(F(v),\ldots,F^{(k)}(v)\right)=J\right).
\end{align*}
The probability on the right hand side  can  be derived  by counting the functions in $\mathfrak{F}(\ds n)$ that lead to the path $J$. Since $J$ determines the images of exactly $k$ elements to be $i_1,\ldots,i_k$, there are 
\begin{align*}
\frac{(n-k)!}{\prod_{\ell\notin\{i_1,\ldots,i_k\}} d_\ell! \prod_{j=1}^k (d_{i_j}-1)!}
\end{align*}
possible ways to choose the remaining images. The assertion follows after dividing by the total number $n!/ \prod_{\ell=1}^n d_\ell!$ of elements in $\mathfrak{F}(\ds n)$.
\end{proof}
\begin{lemma}\label{lem_g}
Let $g_n:[n]\rightarrow [0,\infty)$ be defined as 
\begin{align*}
g_n(k)=\frac {k!} {\ff n k} \sum_{i_1<\ldots<i_k} \prod_{j=1}^k d_{i_j}
\end{align*}
where the summation is taken over all $(i_1,\ldots,i_k)\in[n]^k$ with $i_1<\ldots<i_k$.
Then
\begin{align*}
\Prob(\SL_n(v)>k) = g_n(k) - \frac{ k d_v}{n-k+1} \Prob(\SL_n(v) >k-1),\quad k\geq 2,\,v\in [n].
\end{align*}
\end{lemma}
\begin{proof}
Let $\widetilde{J}_k=\{ (j_1,\ldots,j_k)\in [n]^k : j_\ell \neq j_m \text{ for } \ell \neq m\}$. Lemma \ref{lem_dist_SL} implies
\begin{align}\label{pf_lem_g}
\Prob(\SL_n(v)>k)=\frac 1 {\ff n k} \sum_{(i_1,\ldots, i_k)\in \widetilde{J}_k} \prod_{j=1}^k d_{i_j}
-\frac 1 {\ff n k} \sum_{(i_1,\ldots, i_k)\in \widetilde{J}_k\setminus J_{n,k}(v)} \prod_{j=1}^k d_{i_j}.
\end{align}
The first term equals $g_n(k)$ by matching vectors with equal order statistics. 
 For the second sum note that 
\begin{align*} 
 \widetilde{J}_k\setminus J_{n,k}(v)=\bigcup_{j=1}^k\left\{(i_1,\ldots,i_k)\in [n]^k : i_j=v, (i_1,\ldots,i_{j-1},i_{j+1},\ldots,i_k)\in J_{n,k-1}(v)\right\}.
\end{align*} 
Hence,
\begin{align*}
\sum_{(i_1,\ldots, i_k)\in \widetilde{J}_k\setminus J_{n,k}(v)} \prod_{j=1}^k d_{i_j}
=k d_v \sum_{(i_1,\ldots, i_{k-1})\in J_{n,k-1}(v)} \prod_{j=1}^k d_{i_j}
\end{align*}
and the assertion follows from Lemma \ref{lem_dist_SL}.
\end{proof}
Note that the previous Lemma in particular yields the following bounds:
\begin{align}\label{bounds_sl_gn}
\frac{n-k}{n-k+(k+1)d_v} g_n(k+1) \leq \Prob(\SL_n(v)>k) \leq \frac{n-k+1}{n-k+1+kd_v} g_n(k). 
\end{align}
Thus we can focus on the asymptotic behaviour of $g_n(k)$ for $k=\Theta(\sqrt{n/\sigma^2})$ instead.  However, since we need some large deviation bounds in later proofs, we formulate the following lemmas so as to cover a wider range for $k$ than necessary for \cref{thm_sl_a+}. 

The first step is to transform the sum in $g_n(k)$ into a probability that is covered by  Poission approximation. To this end let
\begin{align*}
\alpha=\alpha(n,k)=\frac k n.
\end{align*}
Then $g_n(k)$ can be rewritten as follows:
\begin{align}\label{eq_gn_1}
g_n(k)=\frac{k!}{\langle n \rangle_k \alpha^k}  \prod_{j=1}^n (\alpha d_j +1) \sum_{i_1<\ldots<i_k} \prod_{j=1}^k \frac{\alpha d_{i_j}}{\alpha d_{i_j}+1} \prod_{\ell\in [n]\setminus\{i_1,\ldots,i_k\}} \frac 1 {\alpha d_\ell +1}.
\end{align}
Now let $B_n$ be binomially $B(n,\alpha)$ distributed. Moreover, let $X_1,\ldots,X_n$ be
independent, Bernoulli distributed random variables with $\Prob(X_i=1)={\alpha d_{i_j}}/(\alpha d_{i_j}+1)$ and let $S_n=X_1+\cdots+X_n$. Then \eqref{eq_gn_1} yields
\begin{align}\label{eq_gn_2}
g_n(k)=(1-\alpha)^{n-k} \prod_{j=1}^n (\alpha d_j +1) \frac{\Prob(S_n=k)}{\Prob(B_n=k)}.
\end{align}
\begin{lemma}\label{lem_asym_lambda}
Let $\lambda=\E[S_n]$, that is
\begin{align*}
\lambda=\sum_{j=1}^n \frac{\alpha d_j}{\alpha d_j+1}
\end{align*} 
with $\alpha=k/n$. Moreover, let $x\wedge y=\min\{x,y\}$. Then 
\begin{align*}
\lambda= k - \frac{k^2 m_2}{n^2} +\bo\left(\frac{k^3 m_3}{n^3}\wedge \frac{k^2 m_2}{n^2}\right).
\end{align*}
 In particular,
$\lambda-k=\bo\left(k^2 m_2/n^2\right)$.
\end{lemma}
\begin{proof}Note that for $x\geq 0$
\begin{align*}
\frac x {x+1}=x-x^2+\frac{x^3}{x+1}=x-x^2+\bo(x^3\wedge x^2).
\end{align*}
Using this bound in the definition of $\lambda$ yields the assertion.
\end{proof} 
Next we apply Chen-Stein  Poisson approximation to obtain the following result:
\begin{lemma}\label{lem_g_1} Let $\lambda$ be as in the previous lemma.
Then, for  $k=o\left(\left(n^2 / m_2\right)^{2/3}\right)$,  
\begin{align*}
g_n(k)=(1-\alpha)^{n-k} \left(\prod_{j=1}^n (\alpha d_j +1)\right) \mathrm{e}^{k-\lambda} \left(\frac{\lambda}{k}\right)^k \left(1+ \bo\left(\ \frac{k^{3/2} m_2}{n^2} \right)\right).
\end{align*}
\end{lemma}
\begin{proof} A standard Chen-Stein bound for Poisson approximation, such as in Barbour, Holst and Janson~\cite[Equation (1.23)]{BHJ}, implies
\begin{align*}
\left|\Prob(S_n=k)-\mathrm{e}^{-\lambda}\frac {\lambda^k}{k!}\right|&\leq \frac 1 \lambda \sum_{j=1}^n \left(\frac{\alpha d_j}{\alpha d_j+1}\right)^2\leq \frac{\alpha^2 m_2}{\lambda}=\frac{k^2 m_2}{\lambda n^2} ,\\
\left|\Prob(B_n=k)-\mathrm{e}^{-k}\frac {k^k}{k!}\right|&\leq \frac n k \alpha^2 = \frac k n.
\end{align*}
It only remains to transform these into relative error bounds. Note that Stirling's approximation yields
\begin{align*}
\mathrm{e}^{-k}\frac {k^k}{k!}=\Theta\left(\frac 1 {\sqrt k }\right),\qquad \mathrm{e}^{-\lambda}\frac {\lambda^k}{k!}=\Theta\left(\frac {\mathrm{e}^{k-\lambda}}{\sqrt k } \left(\frac \lambda k \right)^k\right).
\end{align*}
As formally shown in Lemma \ref{lem_xxx} below,
 $\mathrm{e}^{k-\lambda} (\lambda/k)^k=1+o(1)$.
Hence, since Lemma~\ref{lem_asym_lambda} implies $\lambda\sim k$,
\begin{align*}
\Prob(S_n=k)&=\mathrm{e}^{-\lambda}\frac{\lambda^k}{k!} \left(1+\bo\left(\frac{k^{3/2} m_2}{n^2} \right)\right),\\
\Prob(B_n=k)&=\mathrm{e}^{-k}\frac {k^k}{k!} \left(1+\bo\left(\frac{k^{3/2}} n\right)\right).
\end{align*}
Therefore \eqref{eq_gn_2} implies the assertion.
\end{proof}

\begin{lemma}\label{lem_xxx} Let  $k=o\left(\left(n^2/m_2\right)^{2/3}\right)$. Then
$\mathrm{e}^{k-\lambda} \left( \lambda /k \right)^k=1+\bo\left(k^3m_2^2/n^4\right)$.
\end{lemma}
\begin{proof} First note that $k-\lambda=\bo(k^2m_2/n^2)$ by Lemma \ref{lem_asym_lambda}.
In particular $k-\lambda=o(\sqrt k)$ by assumption on $k$. Hence, since $\log(1-x)=-x+\bo(x^2)$ as $x\rightarrow 0$, 
\begin{align*}
\frac \lambda k =\exp\left(-\frac{k-\lambda} k +\bo\left(\frac{(k-\lambda)^2}{k^2}\right)\right). 
\end{align*}
Thus
\begin{align*}
\mathrm{e}^{k-\lambda}\left(\frac \lambda k\right)^k = \exp\left(\bo\left(\frac {(k-\lambda)^2} k \right)\right)
\end{align*}
and the assertion follows using the above bound on $k-\lambda$.
\end{proof}
\begin{lemma}\label{lem_asym_g}
Assume  $k=o\left( (n^2/m_2)^{2/3}\wedge (n^3/m_3)^{1/3}\right)$. Then
\begin{align*}
g_n(k)=\exp\left(-\frac{k^2 \sigma^2}{2n}\right)\left(1+\bo\left(\frac{k^{3/2}m_2} {n^2}+ \frac{k^3 m_3}{n^3}\right)\right).
\end{align*}
\end{lemma}
\begin{proof}First note that  Lemmas~\ref{lem_g_1} and~\ref{lem_xxx} yield  
\begin{align}\label{pf_lem_asym_g}
g_n(k)=(1-\alpha)^{n-k}\left(\prod_{j=1}^n (\alpha d_j +1)\right) \left(1+\bo\left(\frac{k^{3/2} m_2} {n^2}\right)\right).
\end{align}
By expanding $\log(1+x)$  and using $\alpha=k/n$  we find
\begin{align*}
(1-\alpha)^{n-k}=\exp\left( -\alpha (n-k)- \frac{\alpha^2(n-k)} 2 +\bo\left(\alpha^3 n \right)\right)=\exp\left(-k + \frac{k^2} {2n} +\bo\left(\frac{k^3} {n^2}\right)\right)  
\end{align*}
and
\begin{align*}
\prod_{j=1}^n (\alpha d_j +1)=\exp\left(k-\frac{k^2 m_2}{2 n^2} +\bo\left(\frac{k^3 m_3}{n^3}\right)\right).
\end{align*}
Hence the assertion follows from~\eqref{pf_lem_asym_g} and $\sigma^2=m_2/n -1$, noting that the error term tends to 0.
\end{proof}
 As a last step before proving Proposition \ref{thm_sl_a+}, note the following:
\begin{lemma}\label{lem_m3} Assumptions \eqref{a2} and \eqref{a+} imply
$m_3(\ds n) = o\left( (n\sigma^2(\ds n))^{3/2}\right)$.
\end{lemma}
\begin{proof}
First note that
\begin{align}\label{pf_lem_m3}
\Delta n \sigma^2\geq \sum_{v\in[n]}d_v (d_v-1)^2=m_3-2m_2+n=m_3-2n\sigma^2 - n.
\end{align}
Now \eqref{a2} yields $(\Delta+2)n\sigma^2=o\left((n\sigma^2)^{3/2}\right)$, whereas \eqref{a+} ensures $n=o\left( (n\sigma^2)^{3/2}\right)$. Therefore, \eqref{pf_lem_m3} implies the assertion.
\end{proof}
\begin{proof}[Proof of \cref{thm_sl_a+}] Let $x>0$ and let $k=\lfloor x \sqrt{n/\sigma^2}\rfloor$. Note that Assumption \eqref{a2} combined with \eqref{bounds_sl_gn} yields
\begin{align*}
\Prob(\SL_n(v_n)>x \sqrt{n/\sigma^2})=g_n(k) +o(1).
\end{align*}
Moreover, note that \cref{lem_m3} implies $k=o\left( (n^3/m_3)^{1/3}\right)$, whereas \eqref{a1} and \eqref{a+} imply $k=o\left( (n^2/m_2)^{2/3}\right)$. Therefore \cref{lem_asym_g} yields the assertion.
\end{proof}

\subsection{Moment convergence.} Next up is the proof of \cref{thm_moments} under assumption \eqref{a+}, that is:
\begin{proposition}\label{thm_moments_a+}
Assume \eqref{a0}, \eqref{b1}, \eqref{b2} and \eqref{a+}. Let $X$ be  standard Rayleigh distributed. Then
\begin{align*}
\lim_{n\rightarrow\infty} \E\left[\left(\frac{\SL_n(v_n)}{\sqrt{n /\sigma^2(\ds n)}}\right)^p\right]=\E[X^p],\quad p\geq 1.
\end{align*}
In particular, $\E[\SL_n(v_n)]\sim \sqrt{\frac {\pi n} {2\sigma^2(\ds n)}}$ and $\V(\SL_n(v_n))\sim \frac{4-\pi} {2\sigma^2(\ds n)} n $.
\end{proposition}
In preparation for the proof of \cref{thm_moments_a+}, we note the following.
\begin{lemma}\label{lem_m3_moments}
Assumptions \eqref{b2} and \eqref{a+} imply $m_3(\ds n) = o\left( \left(n\sigma^2(\ds n)/\log n\right)^{3/2}\right)$.
\end{lemma}
\begin{proof}
Same as for \cref{lem_m3} up to some obvious changes.
\end{proof}

\begin{proof}[Proof of \cref{thm_moments_a+}]   Let $X_n=\SL_n(v_n)/\sqrt{n/\sigma^2}$ and let $X$ be standard Rayleigh distributed. First note that if $X_n$ converges in distribution to $X$ and
\begin{align}\label{cond_p_mom}
\sup_{n\in\N} \E[X_n^p]<\infty\quad\text{ for all }p\geq 1
\end{align}
then $\E[X_n^p]\rightarrow\E[X]$, since \eqref{cond_p_mom} and Markov's inequality imply that $(X_n^p)_{n\geq 0}$ is uniformly integrable. Hence, by \cref{thm_sl_a+} it is sufficient to show \eqref{cond_p_mom}.

To this end, note that \eqref{bounds_sl_gn} and Lemma \ref{lem_asym_g} imply for every $C>0$
\begin{align}
\Prob\left(X_n>x\right)\leq C^\prime \exp\left(-\frac{x^2} 2\right),\quad x\in\left[0,C\sqrt{\log n}\right], \label{pf_thm_moments_2}
\end{align}
for some constant $C^\prime$ which only depends on $C$. In particular, since $X_n\leq n$, 
\begin{align*}
\E\left[X_n^p \Ind_{\{X_n>C_p \sqrt{\log n}\}}\right]\leq n^p \Prob\left(X_n>C_p\sqrt{\log n}\right)=\bo(1)
\end{align*}
for $C_p=\sqrt{2p}$. Therefore
\begin{align*}
\E\left[X_n^p\right]=\int_0^{C_p\sqrt{\log n}} \Prob\left(X_n^p>x\right) \mathrm{d}x +\bo(1),
\end{align*}
which yields the assertion by \eqref{pf_thm_moments_2}.
\end{proof}
\subsection{Joint limit for tail- and six-length}  In this section we prove \cref{thm_taillength} under the additional assumption \eqref{a+}, that is:
\begin{proposition}\label{thm_taillength_a+} Let $X$ and $U$ be independent, $U$ be uniformly distributed on $[0,1]$, and $X$ be Rayleigh distributed. 
Assume \eqref{a0}, \eqref{b1}, \eqref{b2} and \eqref{a+}. Then
\begin{align*}
\left(\frac{\SL_n(v_n)}{\sqrt{n /\sigma^2(\ds n)}}\, ,\,\frac{\TL_n(v_n)}{\sqrt{n /\sigma^2(\ds n)}}\right)\stackrel{d}{\longrightarrow} (X, UX).
\end{align*}
\end{proposition}

The joint limit of tail- and six-length will be established in two steps:
\begin{itemize}
\item Show that, conditioned on $\TL_n(v)>0$ and $\SL_n(v)=k$, $\TL_n(v)$ is uniformly distributed on $[k-1]$.
\item Show $\Prob(\TL_n(v)>0)\rightarrow 1$ as $n\rightarrow\infty$.
\end{itemize}
The first observation is true for every degree sequence:
\begin{lemma}\label{lem_unif_tail} Let $\ds n $ be any degree sequence with \eqref{a0}. Let $v$ be such that $\Prob(\TL_n(v)>0)>0$ (i.e.~$d_w>1$ for some $w\neq v$). Then, for every $k\geq 2$,  
\begin{align*}
\Prob(\TL_n(v) = j | \SL_n(v)=k,\TL_n(v)>0) =\frac 1 {k-1},\quad j\in[k-1].
\end{align*}
\end{lemma}
\begin{proof} The assertion is obviously true for $k=2$, since $\TL_n(v)\leq \SL_n(v)-1$ and thus $\TL_n(v)\in\{0,1\}$ if $\SL_n(v)=2$. 

Now let $k\geq 3$. It is sufficient to prove
\begin{align}\label{pf_unif_tn_1}
\Prob(\TL_n(v)=i, \SL_n(v)=k)=\Prob(\TL_n(v)=i+1,\SL_n(v)=k),\quad i\in[k-2],
\end{align}
since this implies $\Prob(\TL_n(v) = x | \SL_n(v)=k,\TL_n(v)>0)=\Prob(\TL_n(v) = y | \SL_n(v)=k,\TL_n(v)>0)$ for all $x,y\leq k-1$, yielding an uniform distribution on $[k-1]$.

In order to prove \eqref{pf_unif_tn_1}, let $\mathfrak{F}_{k,i}:=\{f\in \mathfrak{F}(\ds n):\SL_f(v)=k,\TL_f(v)=i\}$. Then \eqref{pf_unif_tn_1} is equivalent to 
\begin{align*}
|\mathfrak{F}_{k,i}|=|\mathfrak{F}_{k,i+1}|,\quad i\in[k-2],
\end{align*}
since the underlying random function $F$ is drawn uniformly at random from  $\mathfrak{F}(\ds n)$. We prove the equality above by finding bijections $\phi_i : \mathfrak{F}_{k,i}\rightarrow \mathfrak{F}_{k,i+1}$. First consider the case $i=k-2$: For $f\in\mathfrak{F}_{k,k-2}$ let $\phi_{k-2}(f)=g$ where $g$ is the function given by
\begin{align*}
g(x)=\begin{cases} f^{(k-1)}(v),&\text{ if } x=f^{(k-3)}(v),\\
f^{(k-2)}(v),&\text{ if } x=f^{(k-2)}(v),\\
f(x),&\text{ otherwise.}
\end{cases}
\end{align*}
The effect of $\phi_{k-2}$ on a functional graph is illustrated in \cref{subfig_phi_k}. 
It is not hard to check that $\phi_{k-2}(\mathfrak{F}_{k,k-2})\subseteq \mathfrak{F}_{k,k-1}$ Note that $\phi_{k-2}$ is invertible by choosing $\phi_{k-2}^{-1}(g):=h$,
\begin{align*}
h(x)=\begin{cases} g^{(k-1)}(v),&\text{ if } x=g^{(k-3)}(v),\\
g^{(k-2)}(v),&\text{ if } x=g^{(k-1)}(v),\\
g(x),&\text{ otherwise.}
\end{cases}
\end{align*}
Thus $\phi_{k-2}$ is a bijection and \eqref{pf_unif_tn_1} follows for $i=k-2$. A similar bijection works for $i<k-2$, as schematically shown in \cref{subfig_phi_i}. Details are left to the reader.
\end{proof}

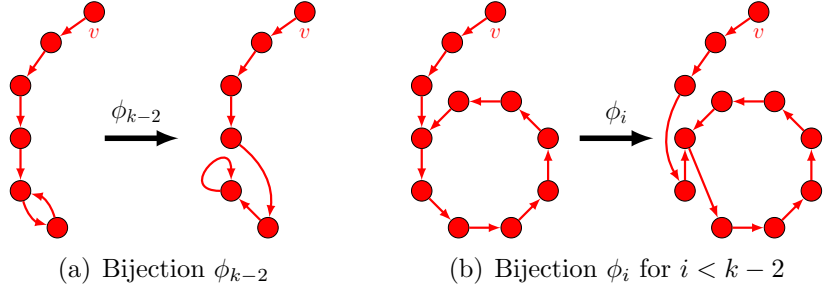
\begin{figure}\centering
\subfigure[Bijection $\phi_{k-2}$\label{subfig_phi_k}]{
\scalebox{0.7}{
\begin{tikzpicture}[every edge/.append style={-{latex},very thick, color=red}]

\node[fill=red,circle,draw] (a) at (1.4,3.4){};
\node[fill=red,circle,draw] (b) at (0.58,2.82){};
\node[fill=red,circle,draw] (c) at (0,2){};
\node[fill=red,circle,draw] (1) at (0,1){};
\node[fill=red,circle,draw] (2) at (0,0){};
\node[fill=red,circle,draw] (3) at (0.7,-0.7){};

\node[draw=none,fill=none] (lab) at (1.4,3){${\color{red} v}$};

\path (a) edge (b);
\path (b) edge (c);
\path (c) edge (1);
\path (1) edge (2);
\path (2) edge [bend right] (3); 
\path (3) edge [bend right] (2); 

\node[fill=red,circle,draw] (a) at (5.4,3.4){};
\node[fill=red,circle,draw] (b) at (4.58,2.82){};
\node[fill=red,circle,draw] (c) at (4,2){};
\node[fill=red,circle,draw] (1) at (4,1){};
\node[fill=red,circle,draw] (2) at (4,0){};
\node[fill=red,circle,draw] (3) at (4.7,-0.7){};

\node[draw=none,fill=none] (lab) at (5.4,3){${\color{red} v}$};

\path (a) edge (b);
\path (b) edge (c);
\path (c) edge (1);
\path (1) edge[bend left] (3);
\path (2) edge [out=180,in=90,looseness=8] (2);
\path (3) edge (2);
\draw[-{latex},line width=3pt] (1.6,1) to (3,1);
\node[draw=none,fill=none] (phi) at (2.2,1.5){{\bf\large$\phi_{k-2}$}};
\end{tikzpicture}
}
}
\hspace{0.5cm}
\subfigure[Bijection $\phi_i$ for $i<k-2$\label{subfig_phi_i}]{
\scalebox{0.7}{
\begin{tikzpicture}[every edge/.append style={-{latex},very thick, color=red}]
\node[fill=red,circle,draw] (a) at (1.4,3.4){};
\node[fill=red,circle,draw] (b) at (0.58,2.82){};
\node[fill=red,circle,draw] (c) at (0,2){};
\node[fill=red,circle,draw] (1) at (0,1){};
\node[fill=red,circle,draw] (2) at (0,0){};
\node[fill=red,circle,draw] (3) at (0.7,-0.7){};
\node[fill=red,circle,draw] (4) at (1.7,-0.7){};
\node[fill=red,circle,draw] (5) at (2.4,0){};
\node[fill=red,circle,draw] (6) at (2.4,1){};
\node[fill=red,circle,draw] (7) at (1.7,1.7){};
\node[fill=red,circle,draw] (8) at (0.7,1.7){};

\node[draw=none,fill=none] (lab) at (1.4,3){${\color{red} v}$};

\path (a) edge (b);
\path (b) edge (c);
\path (c) edge (1);
\path (1) edge (2);
\path (2) edge (3);
\path (3) edge (4);
\path (4) edge (5);
\path (5) edge (6);
\path (6) edge (7);
\path (7) edge (8);
\path (8) edge (1);

\node[fill=red,circle,draw] (at) at (6.4,3.4){};
\node[fill=red,circle,draw] (bt) at (5.58,2.82){};
\node[fill=red,circle,draw] (ct) at (5,2){};
\node[fill=red,circle,draw] (1t) at (5,1){};
\node[fill=red,circle,draw] (2t) at (5,0){};
\node[fill=red,circle,draw] (3t) at (5.7,-0.7){};
\node[fill=red,circle,draw] (4t) at (6.7,-0.7){};
\node[fill=red,circle,draw] (5t) at (7.4,0){};
\node[fill=red,circle,draw] (6t) at (7.4,1){};
\node[fill=red,circle,draw] (7t) at (6.7,1.7){};
\node[fill=red,circle,draw] (8t) at (5.7,1.7){};

\node[draw=none,fill=none] (labt) at (6.4,3){${\color{red} v}$};

\path (at) edge (bt);
\path (bt) edge (ct);

\path (ct) edge[out=240, in=120] (2t);
\path (2t) edge (1t);
\path (1t) edge (3t);

\path (3t) edge (4t);
\path (4t) edge (5t);
\path (5t) edge (6t);
\path (6t) edge (7t);
\path (7t) edge (8t);
\path (8t) edge (1t);
\draw[-{latex},line width=3pt] (3,1) to (4.4,1);
\node[draw=none,fill=none] (phi) at (3.7,1.5){{\bf\large$\phi_i$}};
\end{tikzpicture}
}
}
\caption{Bijections between $\mathfrak{F}_{k,i}$ and $\mathfrak{F}_{k,i+1}$.}
\end{figure}

In order to obtain $\Prob(\TL_n(v)>0)\rightarrow 1$, we first establish the following bound:
\begin{lemma}\label{lem_tn_0_bound} For every $v\in[n]$, as $n\rightarrow\infty$, 
\begin{align*}
\Prob(\TL_n(v)=0)=\bo\left(d_v\E\left[\SL_n(v)/n \right] + d_v n\Prob(\SL_n(v)>n/2)\right).
\end{align*}
\end{lemma}
\begin{proof}Note that $\Prob(\TL_n(v)=0)=\sum_{k=0}^{n-1} \Prob(\TL_n(v)=0,\SL_n(v)=k+1)$ where
\begin{align*}
\Prob(\TL_n(v)=0,\SL_n(v)=k+1)&=\sum_{(v_1,\ldots,v_{k})\in J_{n,k}(v)} \left(\prod_{j=1}^{k} \frac{ d_{v_j}}{n-j+1}\right) \frac {d_v} {n-k}\\
&=\frac {d_v} {n-k} \Prob(\SL_n(v)>k).
\end{align*}
For $k\leq n/2$ we get, uniformly in $k$, $\Prob(\TL_n(v)=0,\SL_n(v)=k+1)=\bo( d_v /n) \Prob(\SL_n(v)>k)$. For $k>n/2$ use $\frac { d_v} {n-k} \Prob(\SL_n(v)>k)\leq { d_v} \Prob(\SL_n(v)>n/2)$. Combining these bounds with $\sum_k \Prob(\SL_n(v)> k)=\E[\SL_n(v)]$ yields the assertion.
\end{proof}
\begin{corollary}\label{lem_tn_0}
Assume \eqref{a0}, \eqref{b1}, \eqref{b2} and \eqref{a+}. Then
$\lim\limits_{n\rightarrow\infty}\Prob(\TL_n(v_n)=0)=0$.
\end{corollary}
\begin{proof}
\cref{thm_moments_a+} and assumption \eqref{b2} imply $d_v\E[\SL_n(v)/n]\rightarrow 0$. Moreover, by \cref{thm_moments_a+} and Markov's inequality
\begin{align*}
\Prob(\SL_n(v)>n/2)=\bo\left((n\sigma^2)^{-p/2}\right),\qquad p\geq 1.
\end{align*}
Thus condition \eqref{a+} implies $d_v n \Prob(\SL_n(v)>n/2)\rightarrow 0$. Therefore \cref{lem_tn_0_bound} yields the assertion.
\end{proof}

\begin{proof}[Proof of \cref{thm_taillength_a+}] Let $U$ be a uniformly on $[0,1]$ distributed random variable that is independent of $(\SL_n(v))_{n\geq 1}$. Moreover, let $\gamma_n=\sqrt{n/\sigma^2(\ds n)}$. 
Then, by Lemma \ref{lem_unif_tail},
\begin{align*}
&\Prob(\SL_n(v)>x\gamma_n,\TL_n(v)>y\gamma_n |\TL_n(v)>0)\\
& = \Prob(\SL_n(v)>x \gamma_n, \lceil U (\SL_n(v)-1)\rceil >y \gamma_n |\TL_n(v)>0).
\end{align*}
Moreover, by Lemma \ref{lem_tn_0} and since $\gamma_n\rightarrow\infty$,
\begin{align*}
&\Prob(\SL_n(v)>x\gamma_n, \lceil U (\SL_n(v)-1)\rceil >y\gamma_n |\TL_n(v)>0)\\
&=\Prob(\SL_n(v)>x\gamma_n, U \SL_n(v) >y\gamma_n)+o(1).
\end{align*}
Finally, Theorem \ref{thm_sixlength} and the independent choice of $U$ yield
\begin{align*}
\Prob(\SL_n(v)>x\gamma_n, U \SL_n(v) >y\gamma_n)\rightarrow \Prob(X>x, U X >y),
\end{align*}
which implies the joint convergence as claimed.
\end{proof}

\section{An extension to cases with small coalescence}\label{sec_small_sig}
In this section we discuss how to extend Theorem \ref{thm_sixlength} to degree sequences with small coalescence, that is sequences with $\sigma^2(\ds n)=\bo(n^{-1/3}\log n)$ and $n\sigma^2(\ds n)\rightarrow\infty$. The key idea is to contract edges incident to vertices with degree $1$
until we obtain a reduced graph that satisfies \eqref{a+}. The six-length of this reduced graph converges to a standard Rayleigh distribution by \cref{thm_sl_a+}. Finally, a concentration argument will allow us to deduce a limit theorem for the original graph.
\begin{defn}\label{def_kernel} Let $\ds n$ be a degree sequence and let $\hat n =\lfloor (n\sigma^2(\ds n))^{4/3} \rfloor$. Let $w$ be a vertex. The \emph{$w$-reduction} of a functional graph  $\mathcal{G}$  is the graph $\mathcal{G}_w$ obtained as follows: If $\hat n\geq n$ let $\mathcal{G}_w=\mathcal{G}$. Otherwise, $\mathcal{G}_w$ is obtained as follows: 
Let $k=n-\hat n$. Let $v_1,\ldots,v_k$ be  $k$ of the degree $1$ vertices in $[n]\setminus\{w\}$, chosen using any canonical method. (Note that there are more than $k$ vertices with degree $1$ by the choice of $k$ and the fact that $n\sigma^2=\sum_j (d_j-1)^2\geq n- \#\{j:d_j=1\}$.) Then do the following for $i=1,\ldots,k$:
\begin{itemize}
\item[(i)]  If $v_iv_i$ is an edge in the graph, then delete $v_iv_i$. Otherwise, replace the two edges $xv_i$ and $v_iy$ incident to $v_i$ by a single edge $xy$;
\item[(ii)] Delete the vertex $v_i$ from the graph.
\end{itemize}
Let $V_w=[n]\setminus\{v_1,\ldots,v_k\}$. Finally, let $\ds {n,w}$ denote the degree sequence of $\mathcal{G}_w$, that is $\ds {n,w}=(d_v)_{v\in V_w}$.
\end{defn}
\begin{remark}
Note that $4/3$ in the definition of $\hat n$ is somewhat arbitrary; the proof works equally well for a range of similar numbers. Also note that $\hat n=o(n)$ for degree sequences with $\sigma^2(\ds n)=o\left(n^{-1/4} \right)$. Finally, note that $\hat n \rightarrow \infty$ as $n\rightarrow\infty$ for any degree sequence with \eqref{a1}.
\end{remark}
\begin{remark}\label{rem_kernel} Suppose $\ds n$ is a degree sequence with \eqref{a1} and \eqref{a2} (or \eqref{b1} and \eqref{b2} respectively), which does not satisfy \eqref{a+}.
Note that $\mathcal{G}_w$ is a functional graph with $\hat n$ vertices and with
\begin{align}\label{eqn_nsigma2}
 \hat n \sigma^2(\ds {n,w})=\sum_{v\in [n]\setminus V_w} (d_v-1)^2=\sum_{v\in [n]} (d_v-1)^2 =n \sigma^2(\ds n).
\end{align}
In particular, $\sigma^2(\ds {n,w})\sim  {\hat n}^{-1/4} $ by the choice of $\hat n$ and therefore $\ds {n,w}$ satisfies \eqref{a+}. Moreover \eqref{eqn_nsigma2} and $\Delta(\ds {n,w})=\Delta(\ds n)$ imply that $\ds {n,w}$ also satisfies \eqref{a1} and \eqref{a2} (or \eqref{b1} and \eqref{b2} respectively).
\end{remark}
\begin{defn}\label{def_blow_up} Let $V^\prime\subset [n]$ and let $G=(V^\prime,E^\prime)$ be a functional graph. An \emph{$n$-extension} of $G$ is a graph $H$ with vertex set $[n]$ which is generated according to the following procedure:
\begin{itemize}
\item[(1)] Start with $V_0=V^\prime$ and $E_0=E^\prime$ and $i=0$.
\item[(2)]{ Let $w$ be the smallest element in $[n]\setminus V_i$. Let $X_w=1$ with probability $1/(|E_i|+1)$ and let $X_w=0$ otherwise. Then do the following:
\begin{itemize}
\item[(a)] If $X_w=1$, add $w$ to the graph as an isolated vertex with a single loop, that is $V_{i+1}=V_i\cup\{w\}$ and $E_{i+1}=E_i\cup\{ ww\}$.
\item[(b)] If $X_w=0$, choose an edge $xy\in E_i$ uniformly at random. Set $V_{i+1}=V_i\cup \{w\}$ and $E_{i+1}=(E_i\setminus \{xy\})\cup\{xv,vw\}$. 
\end{itemize}
\item[(3)] If $V_i=[n]$ set $H=(V_i,E_i)$. Otherwise, increase $i$ by one and return to step 2.
}
\end{itemize}
\end{defn}

\begin{lemma}\label{lem_blow_up} Let $\ds n$ be a degree sequence, $w\in[n]$, and let $\ds {n,w}$ be as in Definition \ref{def_kernel}. If $\mathcal{G}_w$ is a random functional graph with degree sequence  $\ds {n,w}$, then an $n$-extension of $\mathcal{G}_w$ is a random functional graph with degree sequence $\ds n$.
\end{lemma}
\begin{proof}Let $H$ be any functional graph with degree sequence $\ds n$ and let $\mathcal{H}$ denote the $n$-extension of $\mathcal{G}_w$. The claim is that $\Prob(\mathcal{H}=H)=1/|\mathfrak{F}(\ds n)|$.

Since $H$ can only be an $n$-extension of $\mathcal{G}_w$ if $\mathcal{G}_w=H_w$, it is sufficient to show that all possible $n$-extensions of a graph $G$ are equally likely. But since there is exactly one way of choosing edges in (2) throughout the procedure that leads to a particular graph $H$, we have
\begin{align*}
\Prob(\mathcal{H}=H| \mathcal{G}_w=H_w)=\prod_{j=1}^{n-n_w}\frac 1 {n_w+j}
\end{align*}
and the assertion follows.
\end{proof}
\begin{defn}\label{def_polya}
A classical $(a,b)$-P{\'o}lya urn scheme is an urn initialized with $a$ red and $b$ blue balls which evolves in discrete time as follows: In each time step $n$ draw a ball from the urn at random and put it back together with another ball of the same colour.

Let $\mathcal{R}(n,a,b)$ denote the number of red balls after adding $n$ balls to the urn.
\end{defn}
\begin{corollary}\label{coro_SL_polya} Let $\ds {n,w}$ be as in Definition \ref{def_kernel} and
let $\SL_{n,w}(w)$ be the six-length of $w$ in a random functional graph with degree sequence $\ds {n,w}$. Then
\begin{align*}
\SL_n(w)\stackrel{d}{=} \mathcal{R}(n-\hat n,\SL_{n,w}(w),\hat n+1-\SL_{n,w}(w)),
\end{align*}
where $\{\mathcal{R}(n,a,b): a,b,n\in\N_0\}$ is independent of $\SL_{n,w}(w)$ and distributed as in Definition~\ref{def_polya}.
\end{corollary}
\begin{proof}
Identify edges contributing to the six-length $\SL_{n,w}(w)$ with red balls and all other edges (including a 'phantom' edge for step 2a in Definition~\ref{def_blow_up}) with blue balls in a P{\'o}lya urn. Then the dynamics described in Definition~\ref{def_blow_up} is equivalent to the procedure of drawing from a P{\'o}lya urn. Therefore Lemma~\ref{lem_blow_up} implies the assertion.
\end{proof}

\begin{lemma}\label{lem_tail_polya} Let $\mathcal{R}(n,a,b)$ be as in Definition~\ref{def_polya} and let $\mu(n,a,b)=a(1+n/(a+b))$. Then
\begin{align*}
\Prob\left(|\mathcal{R}(n,a,b)-\mu(n,a,b)|\geq t \mu(n,a,b)\right)\leq 2\exp\left(-\frac{t^2 a^2} {8(a+b)}\right).
\end{align*}
\end{lemma}
\begin{proof}
 Let
\begin{align*}
M_{k}:=\frac{\mathcal{R}(k,a,b)}{k+a+b},\quad k\geq 0.
\end{align*}
It is not hard to check that $(M_{k})_{k\geq 0}$ is a martingale. 
Since $|\mathcal{R}(k+1,a,b)-\mathcal{R}(k,a,b)|\leq 1$ and $\mathcal{R}(k,a,b)\leq a+k$, one obtains
\begin{align*}
|M_{k+1}-M_{k}|\leq \frac 2{k+1+a+b}.
\end{align*}
Therefore, the Azuma-Hoeffding inequality yields the assertion.
\end{proof}
We end the section with the missing proofs for Theorems \ref{thm_sixlength}, \ref{thm_moments}, and \ref{thm_taillength}. Note that we may assume w.l.o.g.~that
\begin{align}\label{a-}
\sigma^2(\ds n)=\bo\left(n^{-1/3} \log^2 n\right)\tag{A-},
\end{align}
since the other case is covered by the proofs in Section \ref{sec_proofs}.

\begin{proof}[Proof of \cref{thm_sixlength}] Let $X_n:=\SL_n(v_n)/\sqrt{n/\sigma^2(\ds n)}$ and let $X$ be standard Rayleigh distributed.
The claim is that $X_n$ converges in distribution to $X$.
 By \cref{thm_sl_a+} this holds for degree sequences with \eqref{a+} and thus, we may assume \eqref{a-}.
 
  Let $w=v_n$. Let $\SL_{n,w}(w)$ and $\mathcal{R}(n-\hat n,\SL_{n,w}(w),\hat n+1-\SL_{n,w}(w))$ be as in Corollary~\ref{coro_SL_polya}. Moreover, let $X_{n,w}=\SL_{n,w}(w)/\sqrt{\hat n /\sigma^2(\ds {n,w})}$. Note that
\begin{itemize}
\item[(a)] $X_{n,w}$ converges in distribution to $X$ by \cref{thm_sl_a+} and Remark~\ref{rem_kernel};
\item[(b)]{ $(\SL_{n,w}(w))^2/\hat n\rightarrow \infty$ in probability by (a) and $\sigma^2(\ds {n,w})\sim {\hat n}^{-1/4}$. Hence, using the tail bound in Lemma~\ref{lem_tail_polya} with arbitrary  constant $t>0$,
\begin{align*}
\frac{\mathcal{R}(n-\hat n,\SL_{n,w}(w),\hat n+1-\SL_{n,w}(w))}
{\SL_{n,w}(w)(n+1)/(\hat n+1)}\stackrel{\Prob}{\longrightarrow} 1,
\end{align*}
where $\stackrel{\Prob}{\longrightarrow}$ denotes convergence in probability.
}
\end{itemize} 
Moreover, Corollary~\ref{coro_SL_polya} and $n \sigma^2(\ds n)=\hat n \sigma^2(\ds {n,w})$ (see Remark~\ref{rem_kernel}) imply
\begin{align}\label{pf_thm_sl_x}
X_n\stackrel{d}{=}
\frac{\mathcal{R}(n-\hat n,\SL_{n,w}(w), \hat n+1-\SL_{n,w}(w))}
{\SL_{n,w}(w)(n+1)/(\hat n+1)} X_{n,w} \left(1+o(1)\right),
\end{align}
where $\stackrel{d}{=}$ denotes equality in distribution.
It is not hard to check, e.g.~with Slutsky's Theorem, that \eqref{pf_thm_sl_x}, (a) and (b) imply the assertion. Details are left to the reader.
\end{proof}

\begin{proof}[Proof of \cref{thm_moments}] Let $X_n$, $X_{n,w}$ and $X$ be as in the previous proof. As in the proof of \cref{thm_moments_a+} it is sufficient to show that
\begin{align}\label{pf_thm_moments}
\sup_{n\in\N} \E[X_n^p]<\infty,\quad p\geq 1,
\end{align}
since this bound combined with \cref{thm_sixlength} implies $\E[X_n^p]\rightarrow \E[X^p]$ for all $p\geq 1$. Note that $\sup_n \E[X_{n,w}^p]<\infty$ by \cref{thm_moments_a+} and Remark~\ref{rem_kernel}. 

Now let $A_n:=\{\SL_{n,w}(w)\geq \sqrt{\hat n +1}\}$. With the coupling in \cref{coro_SL_polya}
it is not hard to check that
\begin{align*}
\E[X_n^p |A_n^c]\leq \E[X_n^p | A_n],\quad p\geq 1.
\end{align*}
Thus, since $\Prob(A_n)\rightarrow 1$ by \cref{thm_sixlength}, it is sufficient to show
\begin{align*}
\sup_{n\in\N} \E\left[X_n^p {\bf 1}_{A_n}\right]<\infty,\quad p\geq 1.
\end{align*}
Moreover, by \eqref{pf_thm_sl_x} and $\sup_n \E[X_{n,w}^p]<\infty$ it is sufficient to show that 
\begin{align*}
\sup_{n\in\N}\E\left[\left(
\frac{\mathcal{R}(n-\hat n,\SL_{n,w}(w), \hat n+1-\SL_{n,w}(w))}
{\SL_{n,w}(w)(n+1)/(\hat n+1)}\right)^p {\bf 1}_{A_n}\right]<\infty,\quad p\geq 1,
\end{align*}
which is a consequence of the tail bound in \cref{lem_tail_polya}. Therefore \eqref{pf_thm_moments} holds and the convergence of all moments follows.
\end{proof}
\begin{proof}[Proof of \cref{thm_taillength}] Once again, we may assume w.l.o.g.~that \eqref{a-} holds. Note that we may copy the proof of \cref{thm_taillength_a+} provided we establish
\begin{align}\label{pf_thm_tl}
\Prob(\TL_n(v_n)=0)\rightarrow 0.
\end{align}
Now let $w=v_n$ and let $\TL_{n,w}(w)$ denote the tail-length of $w$ in the $w$-reduction of the functional graph. Note that $\TL_n (w)=0$ if and only if $\TL_{n,w}(w)=0$. 
Since $\Prob(\TL_{n,w}(w)=0)\rightarrow 0$ by \cref{lem_tn_0}, we obtain \eqref{pf_thm_tl}.
Therefore, the assertion follows using the same proof strategy as for \cref{thm_taillength_a+}.
\end{proof}

\noindent
{\bf Acknowledgement.\ } The authors are grateful to Igor Shparlinksi for suggesting this topic of study.

\end{document}